\documentclass[11pt]{amsart}
\usepackage{verbatim} %This is for \comment command
\usepackage{amsmath}
\usepackage{amssymb}
\usepackage{amscd}
\newtheorem{Thm}{Theorem}[section]

\newtheorem{Lemma}[Thm]{Lemma}

\theoremstyle{definition}
\newtheorem{Def}[Thm]{Definition}
\newtheorem{Rmk}[Thm]{Remark}
\newtheorem{Example}[Thm]{Example}

\newtheorem*{question}{Question}

\numberwithin{equation}{section}

\def\bfe{\mathbf{e}}

\def\frakG{\mathfrak{G}}

\def\bbf{\mathbb{F}}

\def\bbr{\mathbb{R}}
\def\bbs{\mathbb{S}}

\def\bbz{\mathbb{Z}}

\def\calC{\mathcal{C}}

\def\calR{\mathcal{R}}

\def\lra{\longrightarrow}

\def\bs{\backslash}

\def\aff{\mathrm{Aff}}
\def\fix{\mathrm{Fix}}
\def\gfix{\mathrm{fix}}

\def\im{\mathrm{im}}

\def\id{\mathrm{id}}

\def\ind{\mathrm{ind}}

\def\Coin{\mathrm{Coin}}

\def\coin{\mathrm{coin}}

\DeclareMathOperator{\Fix}{Fix}

\def\boxit#1{\vbox{\hrule\hbox{\vrule\kern3pt
     \vbox{\kern3pt#1\kern3pt}\kern3pt\vrule}\hrule}}

\def\colr#1{\textcolor{red}{#1}}
\def\colb#1{\textcolor{blue}{#1}}

\begin{document}

\title{An averaging formula for the coincidence Reidemeister trace}
\author{Jong Bum Lee}
\address{Department of Mathematics, Sogang University, Seoul 04107, KOREA}
\email{jlee@sogang.ac.kr}

\author{P. Christopher Staecker}
\address{Department of Mathematics, Fairfield University, Fairfield CT, 06824, USA}
\email{cstaecker@fairfield.edu}

\begin{abstract}
In the setting of continuous maps between compact orientable manifolds of the same dimension, there is a well known averaging formula for the coincidence Lefschetz number in terms of the Lefschetz numbers of lifts to some finite covering space. We state and prove an analogous averaging formula for the coincidence Reidemeister trace. This generalizes a recent formula in fixed point theory by Liu and Zhao.

We give two separate and independent proofs of our main result: one using methods developed by Kim and the first author for averaging Nielsen numbers, and one using an axiomatic approach for the local  Reidemeister trace.  

We also give some examples and state some open questions for the nonorientable case. 
\end{abstract}

%  Math Subject Classifications
\thanks{The first author is supported by Basic Science Researcher Program through the National Research Foundation of Korea (NRF-2016R1D1A1B01006971).}

\maketitle
\tableofcontents

\section{The Reidemeister trace in fixed point theory}

Let $X$ be a compact polyhedron with the universal covering projection
$p:\tilde{X}\to X$ and let $\Pi$ be the group of all covering transformations of $p$.
Let $f:X\to X$ be a self-map and $\tilde{f}:\tilde{X}\to\tilde{X}$
be a lift of $f$.
We define an endomorphism $\phi:\Pi\to\Pi$ given by the rule
\begin{align}\label{endo phi}
\phi(\alpha)\tilde{f}=\tilde{f}\alpha,\quad\forall\alpha\in\Pi.
\end{align}

Let $\calR[\phi]$ denote the set of Reidemeister classes of $\phi$, defined by the twisted conjugacy relation $\alpha \sim \beta$ if and only if there is some $\gamma$ with $\alpha = \gamma^{-1} \beta \phi(\gamma)$. We write $[\alpha]_\phi$ for the Reidmeister class of $\alpha$. When the $\phi$ is understood, we simply write $[\alpha]$.

We can consider the free Abelian group $\bbz\calR[\phi]$ with generators given by the Reidemeister classes. There is a one-to-one correspondence between the Nielsen fixed point classes $\bbf$
and the Reidemeister classes $[\alpha]$. Thus the Nielsen classes $\bbf$ can be indexed by the Reidemeister classes $[\alpha]$; $\bbf=\bbf_{[\alpha]}$.

The Reidemeister number $R(f)=R(\phi)$ is the number of Reidemeister classes.
The Lefschetz number $L(f)$ defined using the homology groups 
is the sum of local indices, which is an important result obtained by Lefschetz. That is,
$$
L(f)=\sum_{[\alpha] \in\calR[\phi]}\ind(f,\bbf_{[\alpha]})\in\bbz.
$$

The Reidemeister trace of $f$ is defined as follows:
\begin{align}\label{trace}
RT(f,\tilde{f})=\sum_{[\alpha]\in\calR[\phi]}\ind(f,\bbf_{[\alpha]})[\alpha]\in\bbz\calR[\phi].
\end{align}
The Reidemeister trace contains information of the Nielsen number and the Lefschetz number.
Indeed, the number of non-zero terms in this sum is the Nielsen number $N(f)$,
and the sum of the coefficients is the Lefschetz number $L(f)$.

Let $\Gamma$ be a finite index normal subgroup of $\Pi$ which is $\phi$-invariant.
Letting $\bar\Pi=\Pi/\Gamma$, we have the following commutative diagram
$$
\CD
1@>>>\Gamma@>>>\Pi@>>>\bar\Pi@>>>1\\
@.@VV{\phi'}V@VV{\phi}V@VV{\bar\phi}V\\
1@>>>\Gamma@>>>\Pi@>>>\bar\Pi@>>>1
\endCD
$$

Let $\bar{X}=\Gamma\bs\tilde{X}$, $p':\tilde{X}\to\bar{X}$ and $\bar{p}:\bar{X}\to X$.
Then we have the following commutative diagram
$$
\CD
\tilde{X}@>{\alpha\tilde{f}}>>\tilde{X}\\
@VV{p'}V@VV{p'}V\\
\bar{X}@>{\bar\alpha\bar{f}}>>\bar{X}\\
@VV{\bar{p}}V@VV{\bar{p}}V\\
X@>{f}>>X
\endCD
$$
where $\alpha\in\Pi$ and $\bar\alpha$ is the image of $\alpha$ in $\bar\Pi$.
Note that $\{\alpha\tilde{f}\mid\alpha\in\Pi\}$ is the set of all lifts on $\tilde{X}$  of $f$,
and $\{\bar\alpha\bar{f}\mid\alpha\in\Pi\}$ is the set of all lifts on $\bar{X}$ of $f$.
%\todo{Add here some background on averaging- where does the original $L(f)$ averaging come from? I see it in Jiang's book with no citations.}
%\sqbox{{\bf Theorem~III.2.12, p.52, B.~Jiang's book.} In my talk some time ago, Jiang attended, I quoted his name for $L(f)$ averaging, and he did not give any comment about this.
%Since then, I have believed that Jiang had formulated it.}
%\todo{For $L(f,g)$ is it due to McCord?}
%\sqbox{He discussed it in his paper, Pacific J. Math., {\bf 154} (1992), 345--368. (I remember that SW Kim read this paper hard long time ago.)
%But I do not think he stated it as an averaging form. Please see p 360, just above Theorem 5.9.
%He states that
%$$
%\frac{1}{|\Phi_1|} \sum_{\beta\in \Phi_2} L(\beta\circ\tilde{f},\tilde{g})
%$$
%is independent of both $\Gamma_i$. But he never said what it is. It is $L(f,g)$!}

It has been known for some time that in this setting the Lefschetz number obeys the following averaging formula:
\[ L(f) = \frac{1}{[\Pi:\Gamma]} \sum_{\bar \beta \in \bar \Pi} L(\bar \beta \bar f). \]
This formula is apparently due to Jiang, and appears in \cite{jiang} from 1983. The averaging formula is interesting in its own right, but also computationally useful. When the finite covering space $\bar X$ is topologically simpler than $X$ itself, the individual terms $L(\bar \beta\bar f)$ may be easy to compute even when $L(f)$ itself is hard to compute.

It is natural to ask if a similar averaging formula holds for the Nielsen number $N(f)$. This is the subject of \cite{KLL-Nagoya}, which shows that the corresponding formula for $N(f)$ does not hold in general. 

Recall from \eqref{endo phi} that the endomorphism $\phi$ was defined using the lift $\tilde{f}$.
Taking projection by $p'$ we have 
$$
\bar\phi(\bar\alpha)\bar{f}=\bar{f}\bar\alpha,\quad\forall\alpha\in\Pi.
$$
If we use other lifts $\alpha\tilde{f}$ and $\bar\alpha\bar{f}$ of $f$, then
we obtain the corresponding endomorphisms $\tau_\alpha\phi$ and $\tau_{\bar\alpha}\bar\phi$ respectively.
(Here, $\tau_\alpha$ is the conjugation by $\alpha$, $\beta\mapsto \alpha\beta\alpha^{-1}$.)
Similarly, the lift $\alpha\tilde{f}$ of $\bar\alpha\bar{f}:\bar{X}\to\bar{X}$
induces an endomorphism $\tau_{\alpha}\phi':\Gamma\to \Gamma$. Hence
$$
RT(\bar\alpha\bar{f},\alpha\tilde{f})=\sum_{[\gamma]_{\tau_\alpha\phi'}\in\calR[\tau_{\alpha}\phi']}
\ind(\bar\alpha\bar{f};\bbf_{[\gamma]_{\tau_\alpha\phi'}})[\gamma]_{\tau_\alpha\phi'},%\in\bbz\calR[\tau_{\alpha}\phi'].
$$
where $\ind$ denotes the fixed point index of a fixed point class.
Furthermore, we have for any $\alpha\in\Pi$,
$$
\CD
1@>>>\Gamma@>>>\Pi@>>>\bar\Pi@>>>1\\
@.@VV{\tau_{\alpha}\phi'}V@VV{\tau_{\alpha}\phi}V@VV{\tau_{\bar\alpha}\bar\phi}V\\
1@>>>\Gamma@>>>\Pi@>>>\bar\Pi@>>>1
\endCD
$$
\begin{comment}
Similarly, the lift $\gamma\alpha\tilde{f}$ ($\gamma\in \Gamma$) of $\bar\alpha\bar{f}:\bar{X}\to\bar{X}$
induces an endomorphism $\tau_{\gamma\alpha}\phi':\Gamma\to \Gamma$. Hence
$$
RT(\bar\alpha\bar{f},\gamma\alpha\tilde{f})=\sum_{[\delta]\in\calR[\tau_{\gamma\alpha}\phi']}
\ind(\bar\alpha\bar{f},\bbf_{[\delta]})[\delta]\in\bbz\calR[\tau_{\gamma\alpha}\phi'].
$$
Furthermore, we have for any $\alpha\in\Pi$ and $\gamma\in \Gamma$,
$$
\CD
1@>>>\Gamma@>>>\Pi@>>>\bar\Pi@>>>1\\
@.@VV{\tau_{\gamma\alpha}\phi'}V@VV{\tau_{\gamma\alpha}\phi}V@VV{\tau_{\bar\alpha}\bar\phi}V\\
1@>>>\Gamma@>>>\Pi@>>>\bar\Pi@>>>1
\endCD
$$
\end{comment}
The diagram induces the following short exact sequence of sets
\begin{equation}\label{calRexact}
\calR[\tau_\alpha\phi']\buildrel{\hat{i}^\alpha}\over\to\calR[\tau_\alpha\phi]\buildrel{\hat{u}^\alpha}\over\to\calR[\tau_{\bar\alpha}\bar\phi]\to1
\end{equation}
and the following sequence of groups
$$
1\to \gfix(\tau_\alpha\phi')\to\gfix(\tau_\alpha\phi)\buildrel{u_\alpha}\over\to\gfix(\tau_{\bar\alpha}\bar\phi),
$$
where $\gfix(\phi)$ denotes the fixed subgroup of the group homomorphism $\phi$.

This exactness is used in \cite{KLL-Nagoya} to give a convenient relabeling of the fixed point classes. Typically the fixed point classes of $f$ are labeled by the Reidemeister classes $\calR[\phi]$ of $\phi$. We may instead relabel them in terms of the Reidemeister classes $\calR[\bar\phi]$ and $\calR[\tau_\alpha \phi']$. This gives the following:
\begin{Lemma}[{\cite[Lemma 2.2]{KLL-Nagoya}}] \label{relabel}
For each $[\bar \beta] \in \calR[\bar\phi]$, choose some $[\beta] \in \calR[\phi]$ with $\hat u^1([\beta]) = [\bar \beta]$. Then we have the disjoint union:
\[ \Fix (f) =\bigsqcup_{[\bar \beta]_{\bar \phi} \in \calR[\bar \phi]}  \bigsqcup_{[\gamma]_{\tau_\beta \phi}\in \im(\hat i^\beta)} p(\Fix(\gamma\beta\tilde f)). \] 
\end{Lemma}

\begin{Rmk}
It will be important for us to relate the quantities $RT(f,\tilde f) \in \calR[\phi]$ and $RT(f,\alpha \tilde f) \in \calR[\tau_\alpha \phi]$ for some $\alpha \in \Pi$. By \eqref{endo phi}, we have that
\[
RT(f,\alpha\tilde{f})=\sum_{[\beta]_{\tau_\alpha\phi}\in\calR[\tau_\alpha\phi]}\ind(f;\bbf_{[\beta]_{\tau_\alpha\phi}})[\beta]_{\tau_\alpha\phi}.%\in\bbz\calR[\tau_\alpha\phi].
\]
In this sum, we note that 
\[ 
\bbf_{[\beta]_{\tau_\alpha\phi}}=p(\fix(\beta(\alpha\tilde{f})))=p(\fix((\beta\alpha)\tilde{f})) = \bbf_{[\beta\alpha]_\phi}
\]
and that multiplication on the right by $\alpha$ gives a one-to-one correspondence $\rho_\alpha: \calR[\tau_\alpha \phi] \to \calR[\phi]$ such that 
\[ \rho_\alpha([\beta]_{\tau_\alpha \phi}) = [\beta \alpha]_\phi. \]
Linearizing (and again using the notation $\rho_\alpha$) gives a one-to-one correspondence 
\[ \rho_\alpha: \bbz \calR[\tau_\alpha \phi] \longleftrightarrow \bbz\calR[\phi]. \]

Thus we have 
\begin{align*}
\rho_\alpha(RT(f,\alpha\tilde f)) 
&= \rho_\alpha \left( \sum_{[\beta]_{\tau_\alpha\phi}\in\calR[\tau_\alpha\phi]}\ind(f;\bbf_{[\beta]_{\tau_\alpha\phi}})[\beta]_{\tau_\alpha\phi} \right)\\
&= \sum_{[\beta\alpha]_{\phi}\in\calR[\phi]}\ind(f;\bbf_{[\beta\alpha]_{\phi}})[\beta\alpha]_{\phi} \\
&= \sum_{[\alpha]_{\phi}\in\calR[\phi]}\ind(f;\bbf_{[\alpha]_{\phi}})[\alpha]_{\phi} 
\end{align*}
and so $RT(f,\tilde f) = \rho_\alpha(RT(f,\alpha \tilde f))$. That is, $RT(f,\tilde f)$ and $RT(f,\alpha \tilde f)$ are the same, except all Reidmeister representatives are multiplied on the right by $\alpha$.

By the same reason, we have 
$$
RT(\bar\alpha\bar{f},\alpha\tilde{f})=\rho_\gamma (RT(\bar\alpha\bar{f},\gamma\alpha\tilde{f})),\quad
\forall\gamma\in \Gamma.
$$
\end{Rmk}

The following averaging formula for Reidemeister traces is the main result of \cite{LZ}. Because the sum of the coefficients in the Reidiemeister trace equals the Lefschetz number, this is a direct generalization of Jiang's averaging formula for $L(f)$. We will give another proof using  arguments from \cite{KLL-Nagoya}.
\begin{Thm}[{\cite[Theorem~4.1]{LZ}}]\label{fixptrt}
For a mapping $f:X\to X$, choose some specific lift $\bar f:\bar X \to \bar X$. Then we have
$$
RT(f,\tilde{f})=\frac{1}{[\Pi:\Gamma]}
\sum_{\bar\alpha\in\Pi/\Gamma} \rho_\alpha\circ\hat{i}^\alpha\left(RT(\bar\alpha\bar{f},\alpha\tilde{f})\right).
$$
\end{Thm}

\begin{proof}
Our proof follows the results and methods from \cite{KLL-Nagoya}. 
For brevity in the following, we will omit subscripts on Reidemeister classes when they are obvious from context. For example $[\gamma]_{\tau_\beta \phi} \in \calR[\tau_\beta\phi]$ will simply be written  $[\gamma] \in \calR[\tau_\beta\phi]$.

Recall that
\[
RT(f,\tilde{f}) =\sum_{[\alpha]\in\calR[\phi]}\ind(f;\bbf_{[\alpha]})[\alpha]\  %\big(\!\in\bbz\calR[\phi]\big)\\
=\sum_{[\alpha]\in\calR[\phi]}\ind(f;p(\fix(\alpha\tilde{f}))[\alpha]_\phi
\]
by the definition of $RT$ and $\bbf_{[\alpha]}$. Relabeling classes as in Lemma \ref{relabel}, for each $[\bar \beta] \in \calR[\bar \phi]$ we fix a preimage $[\beta]\in \calR[\phi]$, and the above becomes:
\begin{align*}
RT(f,\tilde{f})=\sum_{[\bar\beta] \in\calR[\bar\phi]} \left( \sum_{[\gamma] \in\im(\hat{i}^\beta)}\ind(f;p(\fix(\gamma\beta\tilde{f})) 
[\gamma\beta]_\phi \right).
\end{align*}
Rather than the inner sum over elements $[\gamma]$ of $\im(\hat i^\beta)$, we can sum over elements $[\gamma']$ of $\calR[\tau_\beta\phi']$ with $\hat i^\beta([\gamma']) = [\gamma]$. Since   a single $\gamma \in \im(\hat i^\beta)$ may come from several elements in $\calR[\tau_\beta\phi']$, we must insert the fraction below. Letting $F_{\gamma\beta} = p(\fix(\gamma\beta\tilde f))$, we obtain:
\[
RT(f,\tilde{f})=\sum_{[\bar\beta] \in\calR[\bar\phi]} \left( \sum_{[\gamma']\in \calR[\tau_\beta\phi']} \frac{1}{\#(\hat i^\beta)^{-1}([\gamma])} \ind(f;F_{\beta\gamma})
[\gamma\beta]_\phi \right).
\]

By \cite[Lemma 2.1]{KLL-Nagoya}, we also have
$$
\#(\hat i^\beta)^{-1}([\gamma]) = 
[\gfix(\tau_{\bar\beta}\bar\phi):u_{\gamma\beta}(\gfix(\tau_{\gamma\beta}\phi))],
$$
and thus we can rewrite the above as
\begin{align*}
RT(f,\tilde{f})&=\sum_{[\bar\beta] \in\calR[\bar\phi]}\left(\sum_{[\gamma'] \in\calR[\tau_\beta\phi']}
\frac{1}{[\gfix(\tau_{\bar\beta}\bar\phi):u_{\gamma\beta}(\gfix(\tau_{\gamma\beta}\phi))]}
\ind({f};F_{\beta\gamma})[\gamma\beta]_\phi \right).
\end{align*}
Furthermore, we have
\begin{align*}
RT(f,\tilde{f})&=\sum_{\bar\beta\in\Pi/\Gamma}\frac{1}{\#[\bar\beta]}\left(\sum_{[\gamma']\in\calR[\tau_\beta\phi']}
\frac{1}{[\gfix(\tau_{\bar\beta}\bar\phi):u_{\gamma\beta}(\gfix(\tau_{\gamma\beta}\phi))]}
\ind({f};F_{\beta\gamma})[\gamma\beta]_\phi \right)\\
&=\sum_{\bar\beta\in\Pi/\Gamma}\left(\sum_{[\gamma']\in\calR[\tau_\beta\phi']}
\frac{\#u_{\gamma\beta}(\gfix(\tau_{\gamma\beta}\phi))}{\#[\bar\beta]\cdot\#\gfix(\tau_{\bar\beta}\bar\phi)}
\ind({f};F_{\beta\gamma})[\gamma\beta]_\phi \right).
\end{align*}
By \cite[Remark~2.3]{KLL-Nagoya}, we have
$$
[\Pi:\Gamma]=\#[\bar\beta]\cdot\#\gfix(\tau_{\bar\beta}\bar\phi), \ \forall\beta\in\Pi.
$$
Hence
\begin{align*}
RT(f,\tilde{f})&=\frac{1}{[\Pi:\Gamma]}\sum_{\bar\beta\in\Pi/\Gamma}\left(\sum_{[\gamma']\in\calR[\tau_\beta\phi']}
{\#u_{\gamma\beta}(\gfix(\tau_{\gamma\beta}\phi))}
\ind({f};F_{\beta\gamma})[\gamma\beta]_\phi \right).
\end{align*}

Now we remark that the fixed point class $p'(\fix(\gamma\beta\tilde{f}))$ of $\bar\beta\bar{f}$ 
covers by $\bar{p}$ the fixed point class $p(\fix(\gamma\beta\tilde{f}))$ of $f$. 
By \cite[Lemma~2.5]{KLL-Nagoya}, we see that each fixed point of $p(\fix(\gamma\beta\tilde{f}))$ is covered by $\#u_{\gamma\beta}(\gfix(\tau_{\gamma\beta}\phi))$ fixed points of $p(\fix(\gamma\beta\tilde{f}))$. Since the projections $p'$ and $p$ are local homeomorphisms and the index is a local invariant, we have 
\[ \ind(\beta\bar f; p'(\Fix(\gamma\beta\tilde f))) = \#u_{\gamma\beta}(\gfix(\tau_{\gamma\beta}\phi)) \cdot \ind(f;F_{\beta\gamma}), \]
and consequently, we obtain
\begin{align*}
RT(f,\tilde{f})&=\frac{1}{[\Pi:\Gamma]}\sum_{\bar\beta\in\Pi/\Gamma}\left(\sum_{[\gamma']\in\calR[\tau_\beta\phi']}
\ind(\bar\beta\bar{f};p'(\fix(\gamma\beta\tilde{f}))[\gamma\beta]_\phi\  \right).
\end{align*}

We know that $[\gamma\beta]_\phi = \rho_\beta([\gamma]_{\tau_\beta\phi})$, and by the definition of $\gamma'$ we have $[\gamma]_{\tau_\beta\phi} = \hat i^\beta([\gamma']_{\tau_\beta \phi'})$, and thus the above gives:
\begin{align*}
RT(f,\tilde{f})&=\frac{1}{[\Pi:\Gamma]}\sum_{\bar\beta\in\Pi/\Gamma}\left(\sum_{[\gamma']\in\calR[\tau_\beta\phi']}
\ind(\bar\beta\bar{f};p'(\fix(\gamma\beta\tilde{f}))) \rho_\beta(\hat i^\beta([\gamma']_{\tau_\beta\phi'}  )) \right) \\
&=\frac{1}{[\Pi:\Gamma]}\sum_{\bar\beta\in\Pi/\Gamma}\rho_\beta\circ\hat{i}^\beta\left(\sum_{[\gamma']\in\calR[\tau_\beta\phi']}
\ind(\bar\beta\bar{f};p'(\fix(\gamma\beta\tilde{f}))[\gamma']_{\tau_\beta \phi'} \right) \\
&=\frac{1}{[\Pi:\Gamma]}\sum_{\bar\beta\in\Pi/\Gamma} \rho_\beta\circ\hat{i}^\beta\left(RT(\bar\beta\bar{f},\beta\tilde{f})\right).\qedhere
\end{align*}

%
%&=\frac{1}{[\pi:K]}\sum_{\bar\beta\in\pi/K}\left(\sum_{[\gamma]_{\tau_\beta \phi}\in\calR[\tau_\beta\phi']}
%\rho_\beta\left(\ind(\bar\beta\bar{f},p'(\fix(\gamma\beta\tilde{f}))[\gamma]_{\tau_\beta \phi} \right)\right)\\
%&=\frac{1}{[\pi:K]}\sum_{\bar\beta\in\pi/K}\rho_\beta\left(\sum_{[\gamma]_{\tau_\beta \phi}\in\calR[\tau_\beta\phi']}
%\ind(\bar\beta\bar{f},p'(\fix(\gamma\beta\tilde{f}))[\gamma]_{\tau_\beta \phi} \right)\\
%&=\frac{1}{[\pi:K]}\sum_{\bar\beta\in\pi/K}\rho_\beta\circ\hat{i}^\beta\left(\sum_{[\gamma]\in\calR[\tau_\beta\phi']}
%\ind(\bar\beta\bar{f},p'(\fix(\gamma\beta\tilde{f}))[\gamma]\ \left(\in\bbz\calR[\tau_\beta\phi']\right) \right)\\
%
%\end{align*}
\end{proof}

\section{Coincidence theory}

Let $M_1$ and $M_2$ be closed orientable manifolds of equal dimension $n$.
We will generalize the general setup of the previous section to the setting of pairs of mappings $f,g:M_1 \to M_2$ and the coincidence set 
\[
\Coin(f,g)=\{x \in M_1 \mid f(x)= g(x)\}.
\]

Let $p_1:\tilde{M}_1\to M_1$ and $p_2:\tilde{M}_2\to M_2$ be the universal covering projections
with the groups of covering transformations $\Pi_1$ and $\Pi_2$ respectively.
By fixing lifts $\tilde{f}$ and $\tilde{g}$ of $f$ and $g$,
the continuous maps $f$ and $g$ induce homomorphisms $\phi:\Pi_1\to \Pi_2$ and
$\psi:\Pi_1\to \Pi_2$ between the groups of covering transformations as follows.
$$
\phi(\alpha)\tilde{f}=\tilde{f}\alpha,\
\psi(\alpha)\tilde{g}=\tilde{g}\alpha,\quad\forall\alpha\in\Pi_1.
$$
The set of Reidemeister classes $\calR[\phi,\psi]$ is defined as the quotient of $\Pi_2$ by the doubly-twisted conjugacy relation $\alpha \sim \beta$ if and only if there is some $\gamma\in \Pi_1$ with $\alpha = \psi(\gamma)^{-1}\beta\phi(\gamma)$. The Reidemeister class containing $\beta$ will be written as $[\beta] \in \calR[\phi,\psi]$, or when we wish to emphasize the homomorphisms, as $[\beta]_{\phi,\psi}$.

Then it is well known that the coincidence set $\Coin(f,g)$ splits into a disjoint union of coincidence classes. That is,
\begin{align}\label{decomp1}
\Coin(f,g) = \bigsqcup_{[\beta]\in\calR[\phi,\psi]} p_1\!\left(\Coin(\beta\tilde{f},\tilde{g})\right).
\end{align}
Thus any coincidence class $\bbs=p_1\!\left(\Coin(\beta\tilde{f},\tilde{g})\right)$
can be indexed by the Reidemeister class $[\beta]$; $\bbs=\bbs_{[\beta]}$.

\begin{Def}
The \emph{coincidence Reidemeister trace} of $(f,g)$ is defined as follows:
\begin{align}
\label{RT0}
RT(f,\tilde{f},g,\tilde{g})
&=\sum_{[\beta]\in\calR[\phi,\psi]}\ind(f,g;\bbs_{[\beta]})[\beta]\\
&=\sum_{[\beta]\in\calR[\phi,\psi]}\ind(f,g;p_1\!\left(\Coin(\beta\tilde{f},\tilde{g})\right))[\beta]\in\bbz\calR[\phi,\psi].
\notag
\end{align}
\end{Def}
where $\ind$ denotes the coincidence index of a coincidence class.

We remark that the number of non-zero terms in this sum is the Nielsen coincidence number $N(f,g)$,
and the sum of the coefficients is the Lefschetz coincidence number $L(f,g)$.

For $\beta\in\Pi_2$, $\beta\tilde{f}$ is another lift of $f$.
Then the homomorphism induced by $f$ by using the lift $\beta\tilde{f}$
is $\tau_\beta\phi:\Pi_1\to\Pi_2$, where $\tau_\beta$ denotes conjugation by $\beta$.

As in fixed point theory, the quantities $RT(f,\tilde f, g,\tilde g)$ and $RT(f,\beta \tilde f, g,\tilde g)$ are related by the map $\rho_\beta: \calR[\tau_\beta \phi,\psi] \to \calR[\phi,\psi]$ defined by 
\[ \rho_\beta([\gamma]_{\tau_\beta \phi, \psi}) = [\gamma\beta]_{\phi,\psi}.\] 
It can be verified as in fixed point theory that:
\begin{equation}\label{coinsame}
RT(f,\tilde f, g, \tilde g) = \rho_\beta(RT(f,\beta \tilde f, g, \tilde g)).
\end{equation}

%\begin{align}
%\label{RT}
%RT(f,\beta\tilde{f};g,\tilde{g})
%&=\sum_{[\gamma]\in\calR[\tau_\beta\phi,\psi]}
%\ind(f,g;\bbs_{[\gamma]})[\gamma]\\
%&=\sum_{[\gamma]\in\calR[\tau_\beta\phi,\psi]}
%\ind(f,g;p_1\!\left(\Coin(\gamma(\beta\tilde{f}),\tilde{g})\right)[\gamma]\in\bbz\calR[\tau_\beta\phi,\psi].
%\notag
%\end{align}
%In the summation \eqref{RT}, we note that
%\begin{itemize}
%\item $\bbs_{[\gamma]}=p_1\!\left(\coin(\gamma(\beta\tilde{f}),\tilde{g})\right)
%    =p_1\!\left(\Coin((\gamma\beta)\tilde{f},\tilde{g})\right)$, and
%\item there is a one-to-one correspondence
%$$
%\rho_\beta:[\gamma]\in\calR[\tau_\beta\phi,\psi] \longleftrightarrow [\gamma\beta]\in\calR[\phi,\psi].
%$$
%\end{itemize}
%Under this correspondence, the RHS of the above identity \eqref{RT} becomes
%$$
%\sum_{[\gamma\beta]\in\calR[\phi,\psi]}\ind(f,g;p_1\!\left(\Coin((\gamma\beta)\tilde{f},\tilde{g})\right)[\gamma\beta]\left(\in\bbz\calR[\phi,\psi]\right)
%=RT(f,\tilde{f};g,\tilde{g}).
%$$
%That is,
%\begin{align}\label{independent of lift}
%\rho_\beta:RT(f,\beta\tilde{f};g,\tilde{g})\mapsto RT(f,\tilde{f};g,\tilde{g}).
%\end{align}
%Consequently, under the correspondence $\rho_\beta$, the Reidemeister coincidence trace of $f,g$ is independent of the choice of the lift $\tilde{f}$.

%\section{Reidemeister traces under finite regular coverings}

Let $\Gamma_1$ and $\Gamma_2$ be finite index normal subgroups of $\Pi_1$ and $\Pi_2$
so that $\phi(\Gamma_1)\subset \Gamma_2$ and $\psi(\Gamma_1)\subset \Gamma_2$.
For each $\bar\beta\in\Pi_1/\Gamma_2$ and $\beta\in u_2^{-1}(\bar\beta)$,
we have the following commutative diagram
\begin{align}\label{CD}
\CD
1@>>>\Gamma_1@>{i_1}>>\Pi_1@>{u_1}>>\Pi_1/\Gamma_1@>>>1\\
@.@V{\tau_{\beta}\phi'}V{\psi'}V@V{\tau_{\beta}\phi}V{\psi}V@V{\tau_{\bar\beta}\bar\phi}V{\bar{\psi}}V\\
1@>>>\Gamma_2@>{i_2}>>\Pi_2@>{u_2}>>\Pi_2/\Gamma_2@>>>1
\endCD
\end{align}

Moreover the following sequence of groups
$$
1\to \coin(\tau_\beta\phi', \psi')\stackrel{i_1^\beta}{\lra}
\coin(\tau_\beta\phi, \psi)\stackrel{u_1^\beta}{\lra}
\coin(\tau_{\bar\beta}\bar\phi, \bar\psi)
$$
is exact, where $\coin$ denotes the coincidence subgroup of a pair of homomorphisms.
We remark also that $i_2:\Gamma_2\to\Pi_2$ and $u_2:\Pi_2\to\Pi_2/\Gamma_2$
induce maps between the Reidemeister sets $\hat{i}_2^\beta:\mathcal{R}[\tau_\beta\phi',\psi']\to\mathcal{R}[\tau_\beta\phi,\psi]$ and $\hat{u}_2^\beta:\mathcal{R}[\tau_\beta\phi,\psi] \to\mathcal{R}[\tau_{\bar{\beta}}\bar{\phi},\bar{\psi}]$
such that $\hat{u}_2^\beta$ is surjective and $(\hat{u}_2^\beta)^{-1}([\bar{1}])=\mathrm{im}(\hat{i}_2^\beta)$.
That is, the following sequence of sets is exact:
$$
\mathcal{R}[\tau_\beta\phi',\psi']\buildrel{\hat{i}_2^\beta}\over\lra
\mathcal{R}[\tau_\beta\phi,\psi]\buildrel{\hat{u}_2^\beta}\over\lra
\mathcal{R}[\tau_{\bar{\beta}}\bar{\phi},\bar{\psi}]\lra1.
$$

Denoting $\bar{M}_1= \Gamma_1\bs\tilde{M}_1$ and $\bar{M}_2= \Gamma_2\bs\tilde{M}_2$,
we see that $\bar{M}_1$ and $\bar{M}_2$ are regular coverings of $M_1$ and $M_2$ such that
\begin{align*}
&p_1:\tilde{M}_1 \buildrel{p_1'}\over\lra\bar{M}_1\buildrel{\bar{p}_1}\over\lra M_1,\\
&p_2:\tilde{M}_2 \buildrel{p_2'}\over\lra\bar{M}_2\buildrel{\bar{p}_2}\over\lra M_2
\end{align*}
are covering projections. 

Since $\phi(\Gamma_1)\subset \Gamma_2$ and $\psi(\Gamma_1)\subset \Gamma_2$,
the pair of lifts $(\tilde{f},\tilde{g})$ induces the pair of lifts $(\bar{f},\bar{g})$ of $(f,g)$ so that
for any $\bar\beta\in\Pi_1/\Gamma_2$ and $\beta\in u_2^{-1}(\bar\beta)$,
the following diagram is commutative
\begin{equation}\label{coinCD}
\CD
\tilde{M}_1@>{\beta\tilde{f}}>{\tilde{g}}>\tilde{M}_2\\
@VV{p_1'}V@VV{p_2'}V\\
\bar{M}_1@>{\bar\beta\bar{f}}>{\bar{g}}>\bar{M}_2\\
@VV{\bar{p}_1}V@VV{\bar{p}_2}V\\
M_1@>{f}>{g}>M_2
\endCD
\end{equation}

We remark also that
the pair of homomorphisms $(\tau_\beta\phi', \psi')$
are the homomorphisms induced by $(\bar\beta\bar{f},\bar{g})$
by fixing the lifts $(\beta\tilde{f},\tilde{g})$.
Similarly, the pair of homomorphisms $(\tau_{\bar\beta}\bar\phi, \bar\psi)$
are the homomorphisms induced by $({f},{g})$
by fixing the lifts $(\bar\beta\bar{f},\bar{g})$.

Therefore, the coincidence Reidemeister trace of $(\bar\beta\bar{f},\bar{g})$ is by definition
\begin{align}\label{RT lift}
&RT(\bar\beta\bar{f},\beta\tilde{f}, \bar{g},\tilde{g})\\
&=\sum_{[\gamma]\in\calR[\tau_\beta\phi',\psi']}
\ind(\bar\beta\bar{f},\bar{g};p_1'\!\left(\Coin(\gamma(\beta\tilde{f}),\tilde{g})\right)[\gamma]\in\bbz\calR[\tau_\beta\phi',\psi'].
\notag
\end{align}

As in fixed point theory, there is an averaging formula for the Lefschetz coincidence number. The following formula is apparently due to McCord:
\[ L(f,g) = \frac{1}{[\Pi_1:\Gamma_1]} \sum_{\bar \beta \in \Pi_2/\Gamma_2} L(\bar \beta\bar f,\bar g). \]
The summation on the right side above is discussed in \cite[page 360]{mccord92}, and the full formula appears in \cite[page 88]{mccord97}. The corresponding averaging formula for Nielsen numbers does not hold in general, this is studied in \cite{KL}.

%If $\gamma\sim\psi(\delta)\gamma\tau_\beta\phi(\delta)^{-1}$,
%then
%\begin{align*}
%\gamma&\sim \psi(\delta)(\gamma\beta)\phi(\delta)^{-1}\beta^{-1},\\
%\gamma\beta&\sim \psi(\delta)(\gamma\beta)\phi(\delta)^{-1},\\
%\gamma\beta\alpha^{-1}&\sim \psi(\delta)(\gamma\beta\alpha^{-1})\tau_\alpha\phi(\delta)^{-1}.
%\end{align*}
%Hence we can define a bijection
%\begin{align*}
%&\rho_{\beta\alpha^{-1}}:\calR[\tau_\beta\phi,\psi]\to\calR[\tau_\alpha\phi,\psi],\quad [\gamma]\mapsto[\gamma\beta\alpha^{-1}],\\
%&\rho_{\beta\alpha^{-1}}:\bbz\calR[\tau_\beta\phi,\psi]\to\bbz\calR[\tau_\alpha\phi,\psi].
%\end{align*}
%When $\alpha=1$ we simply write $\rho_\beta$.

\section{Main result}

The following averaging formula for the coincidence Reidemeister trace is our main result:
\begin{Thm}\label{mainthm}
For maps $f,g:M_1\to M_2$, choose some specific lifts $\bar f,\bar g:\bar M_1 \to \bar M_2$. We have
$$
RT(f,\tilde{f},g,\tilde{g})=\frac{1}{[\Pi_1:\Gamma_1]}
\sum_{\bar\beta\in\Pi_2/\Gamma_2} \rho_\beta\circ\hat{i}_2^\beta\left(RT(\bar\beta\bar{f},\beta\tilde{f},\bar{g},\tilde{g})\right).
$$
\end{Thm}

\begin{proof}
The proof follows exactly the structure of the proof of Theorem \ref{fixptrt}, substituting the required results in coincidence theory from \cite{KL}.

%Recall that
%\[
%RT(f,\tilde{f},g,\tilde g) =\sum_{[\alpha]\in\calR[\phi,\psi]}\ind(f,g,\bbs_{[\alpha]})[\alpha]
%=\sum_{[\alpha]\in\calR[\phi,\psi]}\ind(f,g,p(\Coin(\alpha\tilde{f}, \tilde g)))[\alpha]
%\]
%by the definition of $RT$ and $\bbs_{[\alpha]}$. 
Lemma 2.2 of \cite{KL} gives a coincidence version of Lemma \ref{relabel}. Specifically the coincidence classes decompose as follows:
\[ \Coin(f,g) = \bigsqcup_{[\bar \beta]\in \calR[\bar\phi,\bar \psi]} \bigsqcup_{[\gamma]_{\tau_\beta\phi,\psi}\in \im(\hat i_2^\beta)} p_1(\Coin(\gamma\beta\tilde f, \tilde g)). \]
Letting $S_{\gamma\beta} =  p_1(\Coin(\gamma\beta\tilde f, \tilde g))$, we may relabel coincidence classes and we have:
\[
RT(f,\tilde{f},g,\tilde g)=\sum_{[\bar\beta] \in\calR[\bar\phi,\bar \psi]} \left( \sum_{[\gamma]_{\tau_\beta\phi, \psi} \in\im(\hat{i}_2^\beta)}\ind(f,g;S_{\gamma\beta}) 
[\gamma\beta]_{\phi,\psi} \right).
\]
As in the proof of Theorem \ref{fixptrt} we obtain:
%Rather than the inner sum over elements $[\gamma]$ of $\im(\hat i_2^\beta)$, we can sum over elements $[\gamma']$ of $\calR[\tau_\beta\phi',\psi']$ with $\hat i_2^\beta([\gamma']) = [\gamma]$. Since   a single $\gamma \in \im(\hat i_2^\beta)$ may come from several elements in $\calR[\tau_\beta\phi',\psi']$, we must insert the fraction below. We obtain:
\[
RT(f,\tilde{f},g,\tilde g)=\sum_{[\bar\beta] \in\calR[\bar\phi,\bar \psi]} \left( \sum_{[\gamma']\in \calR[\tau_\beta\phi',\psi']} \frac{1}{\#(\hat i_2^\beta)^{-1}([\gamma])} \ind(f,g;S_{\gamma\beta})
[\gamma\beta]_{\phi,\psi} \right).
\]

By \cite[Lemma 2.1]{KL}, we also have
$$
\#(\hat i_2^\beta)^{-1}([\gamma]) = 
[\coin(\tau_{\bar\beta}\bar\phi,\bar \psi):u_{\gamma\beta}(\coin(\tau_{\gamma\beta}\phi,\psi))]
= \frac{\#\coin(\tau_{\bar\beta}\bar\phi,\bar \psi)}{\#u_{\gamma\beta}(\coin(\tau_{\gamma\beta}\phi,\psi))}
$$
and thus we can rewrite the above as
\begin{align*}
RT(f,\tilde{f},g,\tilde g)&=\sum_{[\bar\beta] \in\calR[\bar\phi,\bar \psi]} \left( \sum_{[\gamma']\in \calR[\tau_\beta\phi',\psi']} 
\frac{\#u_{\gamma\beta}(\coin(\tau_{\gamma\beta}\phi,\psi))} {\#\coin(\tau_{\bar\beta}\bar\phi,\bar \psi)}
\ind(f,g;S_{\gamma\beta})
[\gamma\beta]_{\phi,\psi} \right) \\
&=\sum_{\bar\beta \in\Pi_2/K} \left( \sum_{[\gamma']\in \calR[\tau_\beta\phi',\psi']} 
\frac{ \#u_{\gamma\beta}(\coin(\tau_{\gamma\beta}\phi,\psi))} {\#[\bar \beta]\cdot \#\coin(\tau_{\bar\beta}\bar\phi,\bar \psi)}
\ind(f,g;S_{\gamma\beta})
[\gamma\beta]_{\phi,\psi} \right).
\end{align*}

%Furthermore, we have
%\begin{align*}
%RT(f,\tilde{f},g,\tilde g)
%&=\sum_{\bar\beta \in\Pi_2/K} \frac{1}{\#[\bar \beta]} \left( \sum_{[\gamma']\in \calR[\tau_\beta\phi',\psi']} 
%\frac{\#u_{\gamma\beta}(\coin(\tau_{\gamma\beta}\phi,\psi))} {\#\coin(\tau_{\bar\beta}\bar\phi,\bar \psi)}
%\ind(f,g,S_{\gamma\beta})
%[\gamma\beta]_{\phi,\psi} \right).
%\\
%&=\sum_{\bar\beta \in\Pi_2/K} \left( \sum_{[\gamma']\in \calR[\tau_\beta\phi',\psi']} 
%\frac{ \#u_{\gamma\beta}(\coin(\tau_{\gamma\beta}\phi,\psi))} {\#[\bar \beta]\cdot \#\coin(\tau_{\bar\beta}\bar\phi,\bar \psi)}
%\ind(f,g,S_{\gamma\beta})
%[\gamma\beta]_{\phi,\psi} \right).
%\end{align*}

On each Reidemeister class $[\bar\beta]$ of $\calR[\bar\phi,\bar\psi]$,
the group $\Pi_1/\Gamma_1$ acts transitively by the rule $\bar\beta\mapsto \bar\psi(\bar\delta)\bar\beta\bar\phi(\bar\delta)^{-1}$. The isotropy subgroup at $\bar\beta$ is
$$
\left\{\bar\delta\mid \bar\psi(\bar\delta)\bar\beta\bar\phi(\bar\delta)^{-1}=\bar\beta\right\}
=\coin(\tau_{\bar\beta}\bar\phi,\bar\psi).
$$
Thus
$$
\left[\Pi_1:\Gamma_1\right]=\#[\bar\beta]\cdot \#\coin(\tau_{\bar\beta}\bar\phi,\bar\psi),\
\forall [\bar\beta]\in\calR[\bar\phi,\bar\psi],
$$
and so
\[
RT(f,\tilde{f},g,\tilde g) 
=\frac{1}{[\Pi_1:\Gamma_1]} \sum_{\bar\beta \in\Pi_2/K} \left( \sum_{[\gamma']\in \calR[\tau_\beta\phi',\psi']} 
\#u_{\gamma\beta}(\coin(\tau_{\gamma\beta}\phi,\psi))
\ind(f,g;S_{\gamma\beta})
[\gamma\beta]_{\phi,\psi} \right).
\]
Now we remark that the coincidence class $p_1'(\Coin(\gamma\beta\tilde{f},\tilde{g}))$ of $(\bar\beta\bar{f},\bar{g})$
covers by $\bar{p}_1$ the coincidence class $S_{\gamma\beta} = p_1(\Coin(\gamma\beta\tilde{f},\tilde{g}))$ of $(f,g)$.
By Lemma~4.2 of \cite{KL} we see that this cover is ${\#u_1(\coin(\tau_{\gamma\beta}\phi,\psi))}$-fold.
Since the projections $p'_1$ and $p_1$ are orientation preserving local homeomorphisms and the index is a local invariant, we have 
\[ \ind(\beta\bar f,\bar g; p_1'(\Coin(\gamma\beta\tilde f,\tilde g))) = \#u_{\gamma\beta}(\coin(\tau_{\gamma\beta}\phi,\psi)) \cdot \ind(f,g;S_{\gamma\beta}), \]
Therefore, the above identity reduces to:
\begin{align*}
RT(f,\tilde{f},g,\tilde g) 
&=\frac{1}{[\Pi_1:\Gamma_1]} \sum_{\bar\beta \in\Pi_2/K} \left( \sum_{[\gamma']\in \calR[\tau_\beta\phi',\psi']} 
\ind(\bar\beta\bar f,\bar g;S_{\gamma\beta})
[\gamma\beta]_{\phi,\psi} \right) \\
&= \frac{1}{[\Pi_1:\Gamma_1]} \sum_{\bar\beta \in\Pi_2/K} 
\rho_\beta \circ \hat i_2^\beta
\left( \sum_{[\gamma']\in \calR[\tau_\beta\phi',\psi']} 
\ind(\bar\beta\bar f,\bar g;S_{\gamma\beta})
[\gamma']_{\phi',\psi'} \right) \\
&= \frac{1}{[\Pi_1:\Gamma_1]} \sum_{\bar\beta \in\Pi_2/K} 
\rho_\beta \circ \hat i_2^\beta \left( RT(\bar \beta \bar f, \beta \tilde f, \bar g, \tilde g) \right). \qedhere
\end{align*}
\end{proof}

\section{An axiomatic proof}
The averaging formula for the coincidence Reidemeister trace can also be obtained as an application of a uniqueness theorem from \cite{s2}. This uniqueness result was a direct generalization of earlier work in \cite{fps,s1,s2} concerning the local fixed point and coincidence index.
To demonstrate the basic idea, we begin with an axiomatic proof of an averaging formula for the local fixed point index.

Recall that the local fixed point index $\ind(f;U)$ is an integer invariant which is defined for any open set $U$ such that $\Fix(f)\cap U$ is compact. 

Let $M$ be a compact differentiable manifold with universal covering $p:\tilde M \to M$ and a finite covering $\bar p: \bar M \to M$ with the natural map $p':\tilde M \to \bar M$ with $p = p' \circ \bar p$. 

When $U\subset M$ is an open set, and $f:U \to M$ is a continuous map, say $(f;U)$ is \emph{admissible} when $\Fix(f)\cap U$ is compact, and let $\mathcal C(M)$ be the set of all admissible pairs on $M$. If $H:M\times [0,1] \to M$ is a homotopy, we say that $H$ is \emph{admissible in $U$} if the set
\[ \{ (x,t) \mid H(x,t) = x \} \]
is compact in $M\times [0,1]$. 

The following is a result of Furi, Pera, and Spadini \cite{fps}.
\begin{Thm}\label{indexuniqueness}
Let $M$ be a compact differentiable manifold, and let $\iota: \mathcal C(M) \to \mathbb R$ be any function satisfying the following 3 axioms:
\begin{itemize}
\item (Homotopy) If there is a homotopy from $(f;U)$ to $(f';U)$ which is admissible in $U$, then $\iota(f;U) = \iota(f';U)$. 
\item (Additivity) If $U_1, U_2$ are disjoint open subsets of $U$ with $\Fix(f) \cap U \subset U_1 \sqcup U_2$, then 
\[ \iota(f;U_1) + \iota(f;U_2) = \iota(f;U). \]
\item (Normalization) If $c$ is a constant map, then $\iota(c;M) = 1$
\end{itemize}
Then $\iota$ is the local fixed point index.
\end{Thm}

We will use the uniqueness theorem above to obtain the following averaging formula for the local fixed point index.
\begin{Thm}\label{indexaverage}
Let $\ind(f;U)$ be the local fixed point index of an admissible pair, and write $\bar U = p^{-1}(U)$. For any specific lift $\bar f:\bar M \to \bar M$ of $f$, we have
\begin{equation}\label{index-average}
\ind(f;U) = \frac1{[\Pi:\Gamma]} \sum_{\bar \beta \in  \Pi/\Gamma} \ind(\bar \beta\bar f;\bar U). 
\end{equation}
\end{Thm}
\begin{proof}
Before proving formula \eqref{index-average} we should note that when $\Fix(f) \cap U$ is compact in $M$, since $q$ is a covering map we will have $\Fix(\bar \beta \bar f) \cap \bar U$ compact in $\bar M$ for any lift $\bar f$ of $f$, and thus the pair $(\bar f, \bar U)$ is admissible and so the index on the right side of \eqref{index-average} is defined.

For any admissible pair $(f;U)$, let 
\[ \iota(f;U) =  \frac1{[\Pi:\Gamma]} \sum_{\bar \beta \in  \Pi/\Gamma} \ind(\bar \beta \bar f;\bar U), \]
where $\bar f$ is any lift of $f$. It suffices to show that $\iota$ satisfies the three axioms of Theorem \ref{indexuniqueness}.

First we remark that $\iota$ is well defined (independent of the choice of lift $\bar f$). It suffices to show that if $\bar f$ and $\hat f$ are two different choices of lifts, then 
\[ \sum_{\bar \beta \in \Pi/\Gamma} \ind(\bar \beta \bar f;\bar U) = \sum_{\bar \beta \in \Pi/\Gamma} \ind(\bar \beta\hat f; \bar U). \]
There is some $\bar \alpha \in \Pi/\Gamma$ with $\bar \alpha \hat f = \bar f$, and so we have:
\begin{align*}
\sum_{\bar \beta \in  \Pi/\Gamma} \ind(\bar \beta \bar f;\bar U) &= \frac1{[\Pi:\Gamma]} \sum_{\bar \beta \in  \Pi/\Gamma} \ind(\bar \beta \bar \alpha \hat f;\bar U) = \frac1{[\Pi:\Gamma]} \sum_{\bar \beta \in  \Pi/\Gamma} \ind(\bar \beta \hat f;\bar U)
\end{align*}
as desired. Now we show that $\iota$ satisfies the three axioms of Theorem \ref{indexuniqueness}.

For the homotopy axiom, let $H$ be an admissible homotopy from $(f;U)$ to $(f';U)$. Then for any lift $\bar\beta \bar f$ of $f$, the homotopy $H$ will lift to a unique homotopy $ \bar H$ from $\bar\beta \bar f$ to some lift $\bar\beta \bar f'$ of $f'$, and it can be checked that this $\bar H$ is admissible in $\bar U$. 
Thus we have
\begin{align*} 
\iota(f;U) &= \frac1{[\Pi:\Gamma]} \sum_{\bar \beta \in \Pi/\Gamma} \ind(\bar \beta \bar f;\bar U) = \frac1{[\Pi:\Gamma]} \sum_{\bar \beta \in \Pi/\Gamma} \ind(\bar \beta\bar f';\bar U) =
\iota(f';U)
\end{align*}
and so $\iota$ satisfies the homotopy axiom. 

For the additivity axiom, let $(f;U)\in \mathcal C(M)$ with $\Fix(f)\cap U \subset U_1 \sqcup U_2$. Then for any lift $\bar f$ of $f$ we have:
\[ \Fix(\bar f) \cap \bar U \subset \bar p^{-1} (\Fix(f) \cap U) \subset \bar p^{-1}(U_1 \sqcup U_2) = \bar U_1 \sqcup \bar U_2. \]
Thus the additivity axiom for $\ind(\bar \beta \bar f;\bar U)$ gives:
\begin{align*}
\iota(f;U) &=  \frac1{[\Pi:\Gamma]} \sum_{\bar \beta \in  \Pi/\Gamma} \ind(\bar \beta\bar f;\bar U) \\
&=  \frac1{[\Pi:\Gamma]} \sum_{\bar \beta \in  \Pi/\Gamma} (\ind(\bar \beta\bar f;\bar U_1) + \ind(\bar \beta\bar f;\bar U_2 ))\\
&= \frac1{[\Pi:\Gamma]} \sum_{\bar \beta \in  \Pi/\Gamma} \ind(\bar \beta\bar f;\bar U_1) + \frac1{[\Pi:\Gamma]} \sum_{\bar \beta \in  \Pi/\Gamma}\ind(\bar \beta\bar f;\bar U_2) \\
&= \iota(f;U_1) + \iota(f;U_2),
\end{align*}
and thus $\iota$ satisfies the additivity axiom.

For the normalization axiom, let $c:M\to M$ be a constant map. Then any lift $\bar c: \bar M \to \bar M$ is a constant map on $\bar M$, and so $\bar \beta \bar c$ is a constant map for any $\bar \beta$. Thus using the normalization axiom on $\ind(\bar \beta \bar c, \bar M)$ gives:
\[ \iota(c;M) = \frac1{[\Pi:\Gamma]} \sum_{\bar \beta \in  \Pi/\Gamma} \ind(\bar \beta \bar c;\bar M) = \frac1{[\Pi:\Gamma]} \sum_{\bar \beta \in  \Pi/\Gamma} 1 = 1. \]
Thus $\iota$ satisfies the normalization axiom.

Since $\iota$ satisfies the three axioms, it must be the fixed point index.
\end{proof}

If we take $U=M$ above we obtain a new proof of the classical Lefschetz averaging formula, since $\ind(f;M) = L(f)$. 

An averaging formula like the one in Theorem~\ref{indexaverage} can be obtained for the local coincidence index, by substituting the uniqueness theorem for the coincidence index from \cite{s1} in place of Theorem \ref{indexuniqueness}. The uniqueness theorems in \cite{s2} can be used to obtain similar averaging formulas for the local Reidmeister trace and local coincidence Reidemeister trace. We will prove only the most general of these, the local coincidence Reidemeister trace.

Theorem \ref{indexuniqueness} was generalized to the coincidence Reidemeister trace in \cite{s2} as follows: Let $M_1$ and $M_2$ be oriented closed manifolds of the same dimension, let $\tilde M_1$ and $\tilde M_2$ be their universal covers, and let $\mathcal C(M_1,M_2)$ be the set of \emph{admissible tuples} $(f,\tilde f, g, \tilde g; U)$ where $f,g: M_1\to M_2$ are continuous, $\tilde f,\tilde g:\tilde M_1 \to \tilde M_2$ are lifts of $f$ and $g$, and $U$ is an open set with $\Coin(f,g) \cap U$ compact. If $H,G:M_1\times [0,1]\to M_2$ are homotopies, we say the pair $(H,G)$ is admissible in $U$ when the set
\[ \{ (x,t)\in M_1 \times [0,1] \mid H(x,t) = G(x,t) \} \]
is compact.

The local coincidence Reidemeister trace is defined as in Definition \ref{RT0}, but restricting only to the subset $U$:
\[ 
RT(f,\tilde f,g, \tilde g ; U) = \sum_{[\beta]\in \calR[\phi,\psi]} \ind(f|_U,g|_U;\bbs_{[\beta]}|_U) [\beta],
\]
where $f|_U$ and $g|_U$ denote the restrictions of $f$ and $g$ to $U$, and 
\[
\bbs_{[\beta]}|_U = p_1(\Coin(\beta \tilde f|_{\tilde U}, \tilde g|_{\tilde U})), \qquad \tilde U = p_1^{-1}(U). \]

It is clear from the definition that $RT(f,\tilde f,g,\tilde g; M_1) = RT(f,\tilde f,g,\tilde g)$.

The manifolds in \cite{fps,s1,s2} are always assumed to be differentiable. This assumption is made in order to use transversality arguments, but is not necessary. Jezierski proves a ``topological transversality lemma'' in \cite{jez} which provides the same results for topological manifolds. The work in \cite{gs} generalizes all results from \cite{fps,s1,s2} without the differentiability assumption, so we will not require differentiability in this paper. We will also use the normalization axiom used in \cite{gs}, which is weaker than the one from \cite{s2} based on the Lefschetz number. In our normalization property we assume that $M_1$ and $M_2$ have fixed chosen orientations. 

When $S$ is any set, let $\epsilon: \mathbb Z S \to \mathbb Z$ be the augmentation map (sum of coefficients).

\begin{Thm}[\cite{s2}, Theorem 3; \cite{gs}, Theorem 44]\label{rtuniqueness} Let $\sigma$ be a function with domain $\calC(M_1,M_2)$ such that $\sigma(f,\tilde f, g, \tilde g; U) \in \mathbb Z \calR[\phi,\psi]$, where $\phi$ and $\psi$ are the induced homomorphisms of $f$ and $g$ respectively. Assume that $\sigma$ satisfies the following axioms:

\begin{itemize}
\item (Homotopy) If there is a pair of homotopies from $f$ to $f'$ and $g$ to $g'$ which is admissible in $U$, then $\sigma(f,\tilde f, g, \tilde g; U) = \sigma(f', \tilde f', g', \tilde g'; U)$, where $\tilde f'$ and $\tilde g'$ are the unique lifts of $f'$ and $g'$ obtained by lifting the homotopies to $\tilde f$ and $\tilde g$.
\item (Additivity) If $U_1, U_2$ are disjoint open subsets of $U$ with $\Coin(f,g) \cap U \subset U_1 \sqcup U_2$, then
\[ \sigma(f,\tilde f, g, \tilde g; U) = \sigma(f,\tilde f, g, \tilde g; U_1) + \sigma(f,\tilde f, g, \tilde g; U_2) \]
\item (Normalization) Let $c:M_1 \to M_2$ be a constant map with constant value $c\in M_2$, and $g:U\to M_2$ be an orientation preserving embedding with $g(x)=c$. (That is, $g$ carries the positive orientation on $U$ to the positive orientation on $g(U)$.) Then:
\[ \epsilon(\sigma(c,\tilde c, g, \tilde g; U)) = 1. \]
\item (Lift invariance) For any $\alpha, \beta \in \Pi_2$, we have
\[ \epsilon(\sigma(f,\tilde f, g, \tilde g; U)) = \epsilon(\sigma(f,\alpha \tilde f, g,\beta \tilde g; U)) \]
\item (Coincidence of lifts) If $[\alpha]$ appears in $\sigma(f,\tilde f, g, \tilde g; U)$ with nonzero coefficient, then $\alpha \tilde f$ and $\tilde g$ have a coincidence on $p_1^{-1}(U)$.
\end{itemize}
Then $\sigma$ is the local coincidence Reidemeister trace.
\end{Thm}

Our averaging formula for the local coincidence Reidemeister trace is proved similarly to Theorem \ref{indexaverage}.

\begin{Thm}
For any $U \subset M_1$, let $\bar U = \bar p_1^{-1}(U) \subset \bar M_1$ and let $\bar f, \bar g$ be specific chosen lifts of $f,g$. Then we have
\[
RT(f,\tilde{f},g,\tilde{g};U)=\frac{1}{[\Pi_1:\Gamma_1]}
\sum_{\bar \beta \in \Pi_2/\Gamma_2} \rho_\beta\circ\hat{i}_2^\beta\left(RT(\bar\beta\bar{f},\beta\tilde{f},\bar{g},\tilde{g};\bar U)\right),
\]
where $\beta \in \Pi_2$ is any element whose image in $\Pi_2/\Gamma_2$ is $\bar \beta$.
\end{Thm}
\begin{proof}
We will follow the proof of Theorem \ref{indexaverage}. Throughout the proof, for $U \subset M_1$, let $\bar U = \bar p_1^{-1}(U) \subset \bar M_1$. For any tuple $(f,\tilde f, g, \tilde g) \in \calC(M_1,M_2)$, let $\bar f ,\bar g: \bar M_1 \to \bar M_2$ be given by the diagram \eqref{coinCD}. 
Let:
\begin{equation} \label{sigmadef} 
\sigma(f,\tilde f, g, \tilde g; U) = \frac{1}{[\Pi_1:\Gamma_1]}
\sum_{\bar \beta \in \Pi_2/\Gamma_2} \rho_\beta\circ\hat{i}_2^\beta\left(RT(\bar\beta\bar{f},\beta\tilde{f},\bar{g},\tilde{g};\bar U)\right), 
\end{equation}
where $\bar f,\bar g$ are any fixed lifts of $f$ and $g$, 
and we will show that $\sigma$ satisfies the 5 axioms of Theorem \ref{rtuniqueness}. 

First we show that $\sigma$ is well defined, that is, independent of the element $\beta \in \Pi_2$ which projects to $\bar \beta \in \Pi_2/\Gamma_2$. 
It suffices to show that if $\beta$ and $\beta'$ each project to $\bar \beta$, then 
\[
\rho_\beta \circ \hat i_2^\beta (RT(\bar \beta \bar f, \beta \tilde f, \bar g, \tilde g; \bar U)) = \rho_{\beta'} \circ \hat i_2^{\beta'} (RT(\bar \beta \bar f, \beta' \tilde f, \bar g, \tilde g; \bar U)). 
\]

Let $\gamma\in \Gamma_2$ be some element with $ \gamma\beta = \beta'$. Then we must show
\[
\rho_\beta \circ \hat i_2^\beta (RT(\bar \beta \bar f, \beta \tilde f, \bar g, \tilde g; \bar U)) = \rho_{\gamma\beta} \circ \hat i_2^{\gamma\beta} (RT(\bar \beta \bar f, (\gamma\beta)  \tilde f, \bar g, \tilde g; \bar U)),
\]
This follows from the following observation: For $\beta\in\Pi_2$ and $\gamma\in\Gamma_2$, we have the following commutative diagram
$$
\rho_{\gamma\beta}\circ\hat\iota_2^{\gamma\beta}=\rho_\beta\circ\hat\iota_2^\beta\circ\rho_\gamma
$$
$$
\CD
\calR[\tau_{\gamma\beta}\phi',\psi']@>{\hat\iota_2^{\gamma\beta}}>>
\calR[\tau_{\gamma\beta}\phi,\psi]@>{\rho_{\gamma\beta}}>>\calR[\phi,\psi]\\
@VV{=}V@.@AA{\rho_\beta}A\\
\calR[\tau_{\gamma}(\tau_{\beta}\phi)',\psi']@>{\rho_\gamma}>>
\calR[\tau_{\beta}\phi',\psi']@>{\hat\iota_2^\beta}>>\calR[\tau_\beta\phi,\psi]
\endCD
$$
$$
\forall\delta\in\Gamma_2,\quad
\CD
[\delta]@>{\hat\iota_2^{\gamma\beta}}>>[\delta]
@>{\rho_{\gamma\beta}}>>[\delta(\gamma\beta)]=[(\delta\gamma)\beta]\\
@VV{=}V@.@AA{\rho_\beta}A\\
[\delta]@>{\rho_\gamma}>>[\delta\gamma]
@>{\hat\iota_2^\beta}>>[\delta\gamma]
\endCD
$$

Then we have 
\begin{align*} 
\rho_{\gamma\beta} \circ \hat i_2^{\gamma\beta} (RT(\bar \beta \bar f, (\gamma\beta)  \tilde f, \bar g, \tilde g; \bar U)) 
&= \rho_{\beta} \circ \hat i_2^{\beta} \circ \rho_\gamma (RT(\bar \beta \bar f, (\gamma\beta)  \tilde f, \bar g, \tilde g; \bar U)) \\
&= \rho_{\beta} \circ \hat i_2^{\beta} (RT(\bar \beta \bar f, \beta\tilde f, \bar g, \tilde g; \bar U)),
\end{align*}
as required, where the last equality is by \eqref{coinsame}, which holds for the local coincidence Reidemeister trace (this is \cite[Theorem 6]{s2}).

Now we verify that our five axioms hold for $\sigma$.
The homotopy axiom is satisfied exactly as in Theorem \ref{indexaverage}. Any pair of admissible homotopies $(H,G)$ from $f$ to $f'$ and $g$ to $g'$ will uniquely determine a pair of admissible homotopies $(\bar \beta \bar H, \bar G)$ from $\bar \beta \bar f$ to $\bar \beta \bar f'$ and $\bar g$ to $\bar g'$, and also a pair of admissible homotopies from $\tilde f$ to $\tilde f'$ and $\tilde g$ to $\tilde g'$. Then by the homotopy property for the Reidemeister trace we will have:
\begin{align*} 
\sigma(f,\tilde f, g, \tilde g; U) &= \frac{1}{[\Pi_1:\Gamma_1]}
\sum_{\bar \beta \in \Pi_2/\Gamma_2} \rho_\beta\circ\hat{i}_2^\beta\left(RT(\bar\beta\bar{f},\beta\tilde{f},\bar{g},\tilde{g};\bar U)\right) \\
&= \frac{1}{[\Pi_1:\Gamma_1]}
\sum_{\bar \beta \in \Pi_2/\Gamma_2} \rho_\beta\circ\hat{i}_2^\beta\left(RT(\bar\beta\bar{f}',\beta\tilde{f'},\bar{g}',\tilde{g}';\bar U)\right) \\
&= \sigma(f',\tilde f', g', \tilde g';U)
\end{align*}

For the additivity axiom again we follow exactly the argument from Theorem \ref{indexaverage}. We omit the details. 

%\sqbox{
%Since $g$ is an orientation preserving embedding of $U$ and $\bar g$ is a lift of $g$, we have that $\bar g$ is an orientation preserving embedding of $\bar U$. Furthermore it is easy to see that $\bar g(\bar U)$ contains $\bar \beta \bar c$. Then by the normalization property we have
%\[ \epsilon(RT(\bar \beta\bar c, \beta \tilde c, \bar g, \tilde g, \bar U)) = 1, \]
%and thus
%\[ \epsilon(\sigma(c,\tilde c, g, \tilde g, U)) = \frac{1}{[\Pi_1:\Gamma_1]}
%\sum_{\bar \beta \in \Pi_2/\Gamma_2} 1 = 1\]
%as desired.
%}

For the normalization axiom, let $c$ be a constant map and $g$ be an orientation preserving embedding of some small set $U$ with $g(x)=c$ for some $x\in U$. Then there is some $\bar \alpha \in \Pi_2/\Gamma_2$ with $\bar g(\bar x) = \bar \alpha \bar c$. 

We will assume that $U$ is sufficiently small so that $\bar U$ is a union of finitely many components homeomorphic to $U$ by $p_1$. Let $\bar U_0$ be the connected subset of $\bar U$ which contains $\bar x$. Then we have a disjoint union 
\[ \bar U = \bigsqcup_{\bar\gamma  \in \Pi_1/\Gamma_1} \bar\gamma(\bar U_0),\]
and
\begin{align*}
\epsilon(\sigma(c,\tilde c, g, \tilde g; U)) &=  \epsilon\left( \frac{1}{[\Pi_1:\Gamma_1]}
\sum_{\bar \beta \in \Pi_2/\Gamma_2} \rho_\beta\circ\hat{i}_2^\beta\left(RT(\bar\beta\bar{c},\beta\tilde{c},\bar{g},\tilde{g};\bar U)\right)\right) \\
&=   \frac{1}{[\Pi_1:\Gamma_1]}
\sum_{\bar \beta \in \Pi_2/\Gamma_2} \epsilon \left(RT(\bar\beta\bar{c},\beta\tilde{c},\bar{g},\tilde{g};\bar U)\right) \\
&= \frac{1}{[\Pi_1:\Gamma_1]}
\sum_{\bar \beta \in \Pi_2/\Gamma_2} \sum_{\bar \gamma\in \Pi_1/\Gamma_1} \epsilon \left(RT(\bar\beta\bar{c},\beta\tilde{c},\bar{g},\tilde{g};\bar \gamma (\bar U_0))\right).
\end{align*}

The set $\Coin(\bar\beta\bar c, \bar g)\cap \bar \gamma (\bar U_0)$ is nonempty when $\bar g(\bar \gamma \bar x) = \bar \beta \bar c$. Since $\bar g(\bar x) = \bar \alpha \bar c$, this is equivalent to $\bar \psi(\bar \gamma)\bar g(\bar x) = \bar\beta \bar \alpha^{-1} \bar g(\bar x)$, which is to say that $\bar \psi(\bar \gamma) = \bar\beta \bar \alpha^{-1}$.

In the case where $\Coin(\bar\beta\bar c, \bar g)\cap \bar \gamma (\bar U_0)$ is nonempty, it contains a single coincidence point $\bar\gamma\bar x$, and $\bar g$ is an orientation preserving embedding of $\bar \gamma (\bar U_0)$, because covering transformations of covering spaces of orientable manifolds are orientation preserving. Thus by the normalization property for $RT$ we have:
\[ \epsilon \left(RT(\bar\beta\bar{c},\beta\tilde{c},\bar{g},\tilde{g};\bar \gamma (\bar U_0))\right) = \begin{cases} 1 & \text{ if } \bar \psi(\bar \gamma) = \bar\beta \bar \alpha^{-1}, \\
0 & \text{ otherwise.}
\end{cases} \]

Then the above summation gives:
\begin{align*}
\epsilon(\sigma(c,\tilde c, g, \tilde g; U)) &= \frac{1}{[\Pi_1:\Gamma_1]}
\sum_{\bar \beta \in \Pi_2/\Gamma_2} \#\{\bar \gamma \in \Pi_1/\Gamma_1 \mid \bar \psi(\bar \gamma) = \bar\beta \bar \alpha^{-1}\} \\
&= \frac{1}{[\Pi_1:\Gamma_1]}
\sum_{\bar \beta \in \Pi_2/\Gamma_2} \# \bar\psi^{-1}(\{\bar\beta\bar\alpha^{-1}\}) \\
&= \frac{1}{[\Pi_1:\Gamma_1]}
\sum_{\bar \delta \in \Pi_2/\Gamma_2} \# \bar\psi^{-1}(\{\delta\}).
\end{align*}

The set inside the sum has 0 elements when $\delta \in \im\bar\psi$, and otherwise has size $\#\ker\bar\psi$. Thus we have:
\begin{align*}
\epsilon(\sigma(c,\tilde c, g, \tilde g; U)) &= \frac{1}{[\Pi_1:\Gamma_1]}
\#\im\bar\psi \cdot \#\ker\bar\psi = 1
\end{align*}
by the First Isomorphism Theorem.

%\begin{align*}
%\epsilon(\sigma(f,\tilde f, g, \tilde g, M)) &=  \epsilon\left( \frac{1}{[\Pi_1:\Gamma_1]}
%\sum_{\bar \beta \in \lift(\bar f, \bar p_2)} \rho_\beta\circ\hat{i}_2^\beta\left(RT(\bar\beta\bar{f},\beta\tilde{f};\bar{g},\tilde{g},\bar M)\right)\right) \\
%&= \frac{1}{[\Pi_1:\Gamma_1]}
%\sum_{\bar \beta \in \lift(\bar f, \bar p_2)} \epsilon\left(\rho_\beta\circ\hat{i}_2^\beta\left(RT(\bar\beta\bar{f},\beta\tilde{f};\bar{g},\tilde{g},\bar M)\right)\right) \\
%&= \frac{1}{[\Pi_1:\Gamma_1]}
%\sum_{\bar \beta \in \lift(\bar f, \bar p_2)} \epsilon\left(RT(\bar\beta\bar{f},\beta\tilde{f};\bar{g},\tilde{g},\bar M)\right) \\
%&= \frac{1}{[\Pi_1:\Gamma_1]}
%\sum_{\bar \beta \in \lift(\bar f, \bar p_2)} L(\bar \beta \bar f, \bar g) = L(f,g)
%\end{align*}
%where the last equalities are by the normalization property of the Reidemeister trace, and the averaging formula for the Lefschetz number.

For the lift invariance axiom, we have
\begin{align*}
\epsilon(\sigma(f,\alpha \tilde f, g, \gamma \tilde g;U)) &=  \epsilon\left( \frac{1}{[\Pi_1:\Gamma_1]}
\sum_{\bar \beta \in \Pi_2/\Gamma_2} \rho_\beta\circ\hat{i}_2^\beta\left(RT(\bar\beta\bar{f},\beta\alpha\tilde{f},\bar{g},\gamma \tilde{g};\bar U)\right)\right) \\
&= \frac{1}{[\Pi_1:\Gamma_1]}
\sum_{\bar \beta \in \Pi_2/\Gamma_2} \epsilon \left(RT(\bar\beta\bar{f},\beta\alpha\tilde{f},\bar{g},\gamma \tilde{g};\bar U)\right) \\
&= \frac{1}{[\Pi_1:\Gamma_1]}
\sum_{\bar \beta \in \Pi_2/\Gamma_2} \epsilon \left(RT(\bar\beta\bar{f},\beta\tilde{f},\bar{g}, \tilde{g};\bar U)\right) \\
&= \epsilon(\sigma(f, \tilde f, g,  \tilde g; U)).
\end{align*}

Finally we must demonstrate the coincidence of lifts axiom. Say that $[\alpha]$ appears in $\sigma(f,\tilde f, g, \tilde g; U)$ with nonzero coefficient, and we must show that $\alpha \tilde f$ and $\tilde g$ have a coincidence in $p_1^{-1}(U)$. By \eqref{sigmadef}, if $[\alpha]$ has nonzero coefficient in $\sigma(f,\tilde f, g, \tilde g; U)$, then there is some $\bar \beta$ and some $\gamma$ such that $\rho_\beta(\hat i_2^\beta([\gamma])) = [\alpha]$, and $[\gamma]$ has nonzero coefficient in $RT(\bar \beta \bar f, \beta\tilde f, \bar g, \tilde g; \bar U)$. 

Then by the coincidence of lifts property applied to $RT(\bar \beta \bar f, \beta\tilde f, \bar g, \tilde g; \bar U)$, the maps $\gamma \beta \tilde f$ and $\tilde g$ have a coincidence on $\bar U$. But since $\rho_\beta(\hat i_2^\beta([\gamma])) = [\alpha]$, we have $\alpha = \gamma \beta$, and thus $\alpha \tilde f$ and $\tilde g$ have a coincidence on $\bar U$ as desired.

We have shown that $\sigma$ satisfies the 5 axioms of Theorem \ref{rtuniqueness}, and thus it must equal the coincidence Reidmemeister trace.
\end{proof}

\section{Examples}

\begin{Example}
Let $M_1$ be a 3-dimensional orientable flat manifold $\Pi_1\bs\bbr^3$
where $\Pi_1$ is the 3-dimensional orientable Bieberbach group $\frakG_2$ \cite[Theorem~3.5.5]{Wolf}:
$$
\Pi_1=\langle t_1,t_2,t_3,\alpha\mid [t_i,t_j]=1,\
\alpha^2=t_1,\
\alpha t_2\alpha^{-1}=t_2^{-1},\
\alpha t_3\alpha^{_1}=t_3^{-1} \rangle.
$$
We can embed this group into $\aff(\bbr^3)$ by taking $\{\bfe_1, \bfe_2, \bfe_3\}$
as the standard basis for $\bbr^3$ and
$$
t_i=(\bfe_i,I_3)\ (i=1,2,3),\quad
\alpha=\left(\left(\begin{matrix}\frac{1}{2}\\0\\0\end{matrix}\right),
\left(\begin{matrix}1&\hspace{8pt}0&\hspace{8pt}0\\0&-1&\hspace{8pt}0\\0&\hspace{8pt}0&-1\end{matrix}\right)\right).
$$
Then the translation lattice is
$$
\Gamma_1=\Pi_1\cap\bbr^3=\langle t_1,t_2,t_3\rangle=\langle \bfe_1,\bfe_2,\bfe_3\rangle=\bbz^3
$$
and the holonomy group is $\Pi_1/\Gamma_1\cong\bbz_2$.

Let $M_2=\bbr{P}^3$.
Assume that there is a commutative diagram
$$
\CD
\bbr^3@>{\tilde{f}}>{\tilde{g}}>S^3\\
@VV{p_1'}V@VV{p_2'}V\\
T^3=\Gamma_1\bs\bbr^3@>{\bar{f}}>{\bar{g}}>S^3\\
@VV{\bar{p}_1}V@VV{\bar{p}_2}V\\
M_1=\Pi_1\bs\bbr^3@>{f}>{g}>M_2=\bbr{P}^3
\endCD
$$
inducing the following commutative diagram
\begin{align*}
\CD
1@>>>\Gamma_1@>{i_1}>>\Pi_1@>{u_1}>>\Pi_1/\Gamma_1\cong\bbz_2@>>>1\\
@.@V{\tau_{\beta}\phi'}V{\psi'}V@V{\tau_{\beta}\phi}V{\psi}V@V{\tau_{\bar\beta}\bar\phi}V{\bar{\psi}}V\\
1@>>>\Gamma_2=\{1\}@>{i_2}>>\Pi_2@>{u_2}>>\Pi_2/\Gamma_2\cong\bbz_2@>>>1
\endCD
\end{align*}
where $\bar\phi$ is an isomorphism and $\bar\psi$ is a trivial homomorphism.

Let $\beta$ be the antipodal map of $S^3$.
Then $\Pi_2=\langle \beta\rangle$.
Since $\Pi_2$ is abelian, we have $\tau_\beta\phi=\phi$ and $\tau_{\bar\beta}\bar\phi=\bar\phi$.

Remark that since $\psi'$ and $\bar\psi$ are trivial homomorphisms,
it follows that $\psi$ is a trivial homomorphism.
Since $\phi'$ is a trivial homomorphism and $\bar\phi$ is an isomorphism,
it follows that $\phi(\alpha)=\beta$.

A simple computation shows that
\begin{align*}
1&\to\coin(\phi',\psi')=\Gamma_1\to\coin(\phi,\psi)=\Gamma_1\to\coin(\bar\phi,\bar\psi)=\{\bar{1}\}\\
&\to\calR[\phi',\psi']=\{[1]\}\to\calR[\phi,\psi]=\{[1]=[\beta]\}\to\calR[\bar\phi,\bar\psi]=\{[\bar{1}]=[\bar\beta]\}\to1.
\end{align*}
There is only one coincidence class of $(f,g)$:
$$
\bbs_{[1]}=p_1\!\left(\Coin(\tilde{f},\tilde{g})\right).
$$
Hence
$$
RT(f,\tilde{f},g,\tilde{g})=\ind(f,g;\bbs_{[1]})[1].
$$

On the other hand, because $\tau_\beta\phi=\phi$, we have
\begin{align*}
RT(\bar\beta\bar{f},\beta\tilde{f},\bar{g},\tilde{g})
=RT(\bar{f},\tilde{f},\bar{g},\tilde{g})
=\ind(\bar{f},\bar{g};\bbs_{[1]})[1] \in\bbz\calR[\phi',\psi'].
\end{align*}
Thus
\begin{align*}
\frac{1}{[\Pi_1:\Gamma_1]}&\sum_{\bar\beta\in\Pi_2/\Gamma_2}\rho_\beta\circ\hat{i}^\beta\left(RT(\bar\beta\bar{f},\beta\tilde{f}, \bar g, \tilde g)\right)\\
&=\frac{1}{2}
\left(\rho_1\circ\hat{i}_2^{1}\left(RT(\bar{f},\tilde{f},\bar{g},\tilde{g})\right)
+\rho_\beta\circ\hat{i}_2^{\beta}\left(RT(\bar\beta\bar{f},\beta\tilde{f},\bar{g},\tilde{g})\right)\right)\\
&=\frac{1}{2}
\left(\rho_1\circ\hat{i}_2^{1}\left(\ind(\bar{f},\bar{g};\bbs_{[1]})[1]\right)
+\rho_\beta\circ\hat{i}_2^{\beta}\left(\ind(\bar{f},\bar{g};\bbs_{[1]})[1]\right)\right)\\
&=\ind({f},{g};\bbs_{[1]})[1].
\end{align*}

This example demonstrates how the averaging formula for the coincidence Reidemeister trace holds.
Notice also in this example that the averaging formula for the Nielsen coincidence number also holds.
This is expected from from \cite[Theorem~4.6]{KL} since $\coin(\phi,\psi)\subset\Gamma_1$.
\end{Example}

\begin{Example}
Let $M_1=\bbr{P}^3$.
Let $M_2$ be a 3-dimensional orientable flat manifold $\Pi_2\bs\bbr^3$
where $\Pi_2$ is the 3-dimensional orientable Bieberbach group $\frakG_2$:
$$
\Pi_2=\langle t_1,t_2,t_3,\alpha\mid [t_i,t_j]=1,\
\alpha^2=t_1,\
\alpha t_2\alpha^{-1}=t_2^{-1},\
\alpha t_3\alpha^{_1}=t_3^{-1} \rangle.
$$
Let $\Gamma_2=\langle t_1, t_2, t_3 \rangle=\bbz^3$.

Assume that there is a commutative diagram
$$
\CD
S^3@>{\tilde{f}}>{\tilde{g}}>\bbr^3\\
@VV{p_1'}V@VV{p_2'}V\\
S^3@>{\bar{f}}>{\bar{g}}>T^3=\Gamma_2\bs\bbr^3\\
@VV{\bar{p}_1}V@VV{\bar{p}_2}V\\
M_1=\bbr{P}^3@>{f}>{g}>M_2=\Pi_2\bs\bbr^3
\endCD
$$
inducing the following commutative diagram
\begin{align*}
\CD
1@>>>\Gamma_1=\{1\}@>{i_1}>>\Pi_1@>{u_1}>>\Pi_1/\Gamma_1\cong\bbz_2@>>>1\\
@.@V{\tau_{\gamma}\phi'}V{\psi'}V@V{\tau_{\gamma}\phi}V{\psi}V@V{\tau_{\bar\gamma}\bar\phi}V{\bar{\psi}}V\\
1@>>>\Gamma_2@>{i_2}>>\Pi_2@>{u_2}>>\Pi_2/\Gamma_2\cong\bbz_2@>>>1
\endCD
\end{align*}

Let $\beta$ be the antipodal map of $S^3$.
Then $\Pi_1=\langle \beta\rangle$.
Notice that every homomorphism $\xi:\Pi_1\to\Pi_2$ must be a trivial homomorphism.
Otherwise, $\xi(\beta)$  is of the form
$$
\xi(\beta)=t_1^{n_1}t_2^{n_2}t_3^{n_3}\alpha.
$$
Then
$1=\xi(\beta^2)=\xi(\beta)^2=(t_1^{n_1}t_2^{n_2}t_3^{n_3}\alpha)^2=t_1^{2n_1+1}$,
which is impossible.
Thus $\tau_\gamma\phi$ and $\psi$ are trivial, and the remaining homomorphisms are also trivial.

A simple computation shows that
\begin{align*}
1&\to\coin(\phi',\psi')=\Gamma_1\to\coin(\phi,\psi)=\Pi_1\to\coin(\bar\phi,\bar\psi)=\Pi_1/\Gamma_1\\
&\to\calR[\phi',\psi']=\Gamma_2\to\calR[\phi,\psi]=\Pi_2\to\calR[\bar\phi,\bar\psi]=\Pi_2/\Gamma_2\to1.
\end{align*}
The coincidence classes of $(f,g)$ are
$$
\bbs_{\gamma}=p_1\!\left(\Coin(\gamma\tilde{f},\tilde{g})\right),\quad \forall\gamma\in\Pi_2.
$$
Hence
$$
RT(f,\tilde{f},g,\tilde{g})=\sum_{\gamma\in\Pi_2}\ind(f,g;\bbs_{\gamma})\gamma\in\bbz\Pi_2.
$$

On the other hand,
\begin{align*}
\frac{1}{[\Pi_1:\Gamma_1]}&\sum_{\bar\gamma\in\Pi_2/\Gamma_2}\rho_\gamma\circ\hat{i}^\gamma\left(RT(\bar\gamma\bar{f},\gamma\tilde{f},\bar g, \tilde g)\right)\\
&=\frac{1}{2}
\left(\rho_1\circ\hat{i}_2^{1}\left(RT(\bar{f},\tilde{f},\bar{g},\tilde{g})\right)
+\rho_\alpha\circ\hat{i}_2^{\alpha}\left(RT(\bar\alpha\bar{f},\alpha\tilde{f},\bar{g},\tilde{g})\right)\right)\\
&=\frac{1}{2}
\left(\rho_1\circ\hat{i}_2^{1}\left(\sum_{\gamma\in\Gamma_2}\ind(\bar{f},\bar{g};\bbs_{\gamma})\gamma\right)
+\rho_\alpha\circ\hat{i}_2^{\alpha}\left(\sum_{\gamma\in\Gamma_2}\ind(\bar\alpha\bar{f},\bar{g};\bbs_{\gamma})\gamma\right)\right)\\
&=
\sum_{\gamma\in\Gamma_2}\frac{1}{2}\ind(\bar{f},\bar{g};p_1'\!\left(\Coin(\gamma\tilde{f},\tilde{g})\right)\gamma
+\sum_{\gamma\in\Gamma_2}\frac{1}{2}\ind(\bar\alpha\bar{f},\bar{g};p_1'\!\left(\Coin(\gamma\alpha\tilde{f},\tilde{g})\right)\gamma\alpha\\
&=
\sum_{\gamma\in\Gamma_2}\ind({f},{g};p_1\!\left(\Coin(\gamma\tilde{f},\tilde{g})\right)\gamma
+\sum_{\gamma\in\Gamma_2}\ind({f},{g};p_1\!\left(\Coin(\gamma\alpha\tilde{f},\tilde{g})\right)\gamma\alpha.
\end{align*}

Consequently, we have
$$
RT(f,\tilde{f},g,\tilde{g})=\frac{1}{[\Pi_1:\Gamma_1]}\sum_{\bar\gamma\in\Pi_2/\Gamma_2}\rho_\gamma\circ\hat{i}^\gamma\left(RT(\bar\gamma\bar{f},\gamma\tilde{f})\right).
$$
\end{Example}

\section{Open questions in the nonorientable case}

Nielsen coincidence theory for nonorientable manifolds of the same dimension is generally defined in terms of the ``semi-index'' of Dobrenko and Jezierski \cite{dj}. We denote the semi-index of a coincidence class $\bbs$ by $|\ind|(f,g;\bbs)$. The semi-index can be used to define a coincidence Reidemeister trace as in \eqref{RT0}.
\[ |RT|(f,\tilde f,g, \tilde g) = \sum_{[\beta] \in \calR[\phi,\psi]} |\ind|(f, g; \bbs_{[\beta]}) [\beta] \in \bbz\calR[\phi,\psi]. \]

When the manifolds $M_1$ and $M_2$ are orientable, it is well known that $|\ind|(f,g;\bbs) = |\ind(f,g;\bbs)|$, the absolute value of the classical coincidence index. By the definition of $|RT|$, we will have the similar formula
\[ |RT|(f,\tilde f,g,\tilde g) = |RT(f,\tilde f,g,\tilde g)|, \]
where the absolute values on the right denote taking absolute values of each integer coefficient, after grouping together all terms.

It is natural to ask whether an averaging formula will hold for this semi-index Reidemeister trace. 
\begin{question}
Does the following formula hold?
\[ 
|RT|(f,\tilde{f},g,\tilde{g})=\frac{1}{[\Pi_1:\Gamma_1]}
\sum_{\bar\beta\in\Pi_2/\Gamma_2} \rho_\beta\circ\hat{i}_2^\beta\left(|RT|(\bar\beta\bar{f},\beta\tilde{f},\bar{g},\tilde{g})\right)
\]
\end{question}
This formula will hold when the manifolds are orientable- it is not clear if the formula holds when the manifolds are nonorientable. In the following example, the averaging formula does hold.

\begin{Example}
Consider 
$$
\CD
S^2@>{\bar{f}}>>S^2\\
@VV{{p}}V@VV{{p}}V\\
\bbr{P}^2@>{f}>>\bbr{P}^2
\endCD
$$
with $f=\id_{\bbr{P}^2}$ and $\tilde{f}=\id_{S^2}$.
Then the diagram induces the identity endomorphism $\phi$ of the group $\Pi=\{1,\alpha\}$ of transformations of $p$.
Observe that
\begin{itemize}
\item $\calR[\phi,\phi]=\{[1],[\alpha]\}$,
\item $p\left(\Coin(\tilde{f},\tilde{f})\right)=\bbr{P}^2$
and $p\left(\Coin(\alpha\tilde{f},\tilde{f})\right)=\emptyset$,
\item $|\ind|(f,f; p\left(\Coin(\tilde{f},\tilde{f})\right))=1$\newline
(by changing one $f$ to be a small rotation), and\newline 
$|\ind|(f,f; p\left(\Coin(\alpha\tilde{f},\tilde{f})\right))=0.$
\end{itemize}
Then we have
\begin{align*}
|RT|(f,\tilde{f},f,\tilde{f})&=|\ind|(f,f; p(\Coin(\tilde{f},\tilde{f})))[1]
+|\ind|(f,f; p(\Coin(\alpha\tilde{f},\tilde{f})))[\alpha]\\
&=1[1]+0[\alpha]=1[1]\in\bbz\calR[\phi,\phi],\\
N(f,f)&=1.
\end{align*}
On the other hand we can compute $|RT|(\tilde{f},\tilde{f}, \tilde{f},\tilde{f})$ and $|RT|(\alpha\tilde{f},\alpha\tilde{f}, \tilde{f},\tilde{f})$.
Because the manifold involved is an orientable manifold $S^2$, $|RT|$ is the absolute value of the classical $RT$.
For the computation, we remark that $\Gamma=\{1\}$ and so $\phi'$ is trivial, hence $\bar\Pi=\Pi$ and $\bar\phi=\phi$. Indeed, we have
\begin{align*}
|RT|(\tilde{f},\tilde{f}, \tilde{f},\tilde{f})
&=|\ind(\tilde{f},\tilde{f};S^2)|[1]=2[1]\in\bbz\calR[\phi',\phi'],\\
|RT|(\alpha\tilde{f},\alpha\tilde{f}, \tilde{f},\tilde{f})
&=|\ind(\alpha\tilde{f},\tilde{f};\emptyset)|[1]=0[1]\in\bbz\calR[\tau_\alpha\phi',\phi'],
\end{align*}
hence 
\begin{align*}
\frac{1}{2}\left\{
\rho_1\circ\hat\iota_2^1(|RT|(\tilde{f},\tilde{f}, \tilde{f},\tilde{f}))
+\rho_\alpha\circ\hat\iota_2^\alpha(|RT|(\alpha\tilde{f},\alpha\tilde{f}, \tilde{f},\tilde{f}))\right\}
=\frac{1}{2}\left\{2[1]+0[1]\right\}.
\end{align*}
Consequently, we have 
$$
|RT|(f,\tilde{f},f,\tilde{f})
=
\frac{1}{2}\left\{
\rho_1\circ\hat\iota_2^1(|RT|(\tilde{f},\tilde{f}, \tilde{f},\tilde{f}))
+\rho_\alpha\circ\hat\iota_2^\alpha(|RT|(\alpha\tilde{f},\alpha\tilde{f}, \tilde{f},\tilde{f}))\right\}.
$$
Note that in this example the semi-index Nielsen number does not average:
\[
N(f,f)=1\ne \frac{1}{2}(1+0)=\frac{1}{2}(N(\tilde{f},\tilde{f})+N(\alpha\tilde{f},\tilde{f})).
\]
This was expected already by \cite[Theorem~4.5]{KL} as $\coin(\phi,\phi)=\Pi\not\subset\Gamma$.
\end{Example}

The real issue in deriving an averaging formula for the semi-index Reidemeister trace seems to be the semi-index itself. The following question seems likely to be easier, but we still do not have an answer:
\begin{question}
When $M_1$ and $M_2$ are possibly nonorientable, does the following formula hold?
\begin{equation}\label{semiindexavg}
|\ind|(f,g;\bbs) = \frac{1}{[\Pi_1:\Gamma_1]} \sum_{\bar \beta\in\Pi_2 / \Gamma_2} |\ind|(\bar \beta f, \bar g; \bar \bbs) 
\end{equation}
\end{question}

The axiomatic approach is not likely to succeed in answering these questions, since there is no axiomatic formulation of the semi-index. In particular the semi-index is defined only for a coincidence class, not an arbitrary (compact) coincidence set. Therefore the semi-index will not have an additivity property like the one used in \cite{fps,s1,s2,gs}. 

A nonorientable version of Theorem \ref{rtuniqueness} is proved in \cite{gs} using an index with values in $\bbz \oplus \bbz_2$ which is closely related to the semi-index. Roughly, the $\bbz$-part of the index is analogous to the classical index of nondegenerate coincidence points, while the $\bbz_2$-part carries information about any degenerate coincidences. The $\bbz\oplus \bbz_2$-index of a coincidence class is zero if and only if the semi-index is zero, so a Nielsen number using the $\bbz\oplus \bbz_2$-index will agree with the traditional semi-index Nielsen number.
Unlike the semi-index, the $\bbz\oplus \bbz_2$ index can be defined for any coincidence set, and has the appropriate additivity property.

The fact that the values are not integers introduces new complications for a formulation of an averaging formula like \eqref{semiindexavg} for the $\bbz \oplus \bbz_2$-index. In particular the factor of $1/{[\Pi_1:\Gamma_1]}$ needs to be reinterpreted if the sum gives a non-integer value, and also some adjustment must be made for the fact that when the covering spaces are all orientable, the index of all lifts will have values in $\bbz$, while the index of the maps themselves may still have nontrivial $\bbz_2$ part.

Still, the axioms for the $\bbz\oplus\bbz_2$-index may provide a simple way to prove an averaging formula, if it can be properly formulated. Therefore we ask:
\begin{question}
Is there an averaging formula for the $\bbz\oplus \bbz_2$-index analogous to \eqref{semiindexavg}?
\end{question}

\end{document}